\documentclass[11pt]{article}

\input diagram.tex
\usepackage{amsmath} 
\usepackage[mathscr]{eucal}
\usepackage{amssymb} 
\usepackage{theorem}
\usepackage{tikz-cd}
\usetikzlibrary{positioning}

\usepackage{enumerate}

\theoremstyle{change}  
\newtheorem{theorem}{Theorem}[section] 
\newtheorem{lemma}[theorem]{Lemma}  
\newtheorem{proposition}[theorem]{Proposition}

\theorembodyfont{\rmfamily}  

\newenvironment{proof}{\noindent{\bf Proof}\ }{\qed\bigskip}

\renewcommand{\le}{\leqslant} 

\newcommand{\Aut}{\mathrm{Aut}}

\newcommand{\calE}{\mathcal{E}}
\newcommand{\calF}{\mathcal{F}}

\newcommand{\calT}{\mathcal{T}}

\newcommand{\Ind}{\mathrm{Ind}}

\newcommand{\Out}{\mathrm{Out}}

\newcommand{\qed}{\nobreak\hfill
                  \vbox{\hrule\hbox{\vrule\hbox to 5pt
                  {\vbox to 8pt{\vfil}\hfil}\vrule}\hrule}}

\newcommand{\pperm}{T^{\Delta}_{o}}
\newcommand{\Sc}{\mathrm{Sc}}
\title{The group of splendid Morita equivalences of principal $2$-blocks with dihedral and generalised quaternion defect groups}
\author{\c{C}isil Karag{\"u}zel and Deniz Y\i lmaz}
\date{}
\providecommand{\keywords}[1]
{
  \small\smallskip\par	
  \hspace{2ex}\textbf{Keywords:} #1
}
\providecommand{\msc}[1]
{
  \small\smallskip\par	
  \hspace{2ex}\textbf{MSC2020:} #1
}

\begin{document}
\sloppy
\maketitle

\begin{abstract}
Let $k$ be an algebraically closed field of characteristic $2$, let $G$ be a finite group and let $B$ be the principal $2$-block of $kG$ with a dihedral or a generalised quaternion defect group $P$.  Let also $\calT(B)$ denote the group of splendid Morita auto-equivalences of $B$. We show that
\begin{align*}
\calT(B)\cong \Out_P(A)\rtimes \Out(P,\calF)\,,
\end{align*}
where $\Out(P,\calF)$ is the group of outer automorphisms of $P$ which stabilize the fusion system $\calF$ of $G$ on $P$ and $\Out_P(A)$ is the group of algebra automorphisms of a source algebra $A$ of $B$ fixing $P$ modulo inner automorphisms induced by $(A^P)^\times$.
\end{abstract}

\msc{16D90, 20C20.}

\section{Introduction}

Let $k$ be an algebraically closed field of characteristic $p>0$. Let $G$ be a finite group and let $B$ be a block of $kG$ with a maximal Brauer pair $(P,e)$. Let $\calF$ denote the fusion system of $B$ with respect to $(P,e)$. Following \cite[Definition~1.13]{AOV12}, we set $\Aut(P,\calF)$ to be the group of automorphisms of $P$ which stabilize $\calF$ and $\Out(P,\calF)=\Aut(P,\calF)/\Aut_{\calF}(P)$.  Let $\mathrm{Pic}(B)$ denote the Picard group of $B$.  If $M$ is a bimodule inducing a Morita equivalence of $B$, we denote by $[M]$ its isomorphism class in $\mathrm{Pic}(B)$. Let $\calE(B)$ and $\calT(B)$ denote the subgroups of $\mathrm{Pic}(B)$ consisting of Morita equivalences given by bimodules with an endopermutation source and trivial source, respectively. In particular, $\calT(B)$ is the group of splendid Morita (or Puig or source algebra) auto-equivalences of $B$ with product induced by the tensor product over $B$. 

Following \cite{BKL2020}, let $D_k(P,\calF)$ denote the subgroup of the Dade group $D_k(P)$ of $P$ consisting of the isomorphism classes of $\calF$-stable indecomposable endopermutation $kP$-modules. Let also $D_k^t(P,\calF)$ denote its torsion group.

Let $A$ be a source algebra of $B$ and let $\Aut_P(A)$ denote the group of algebra automorphims of $A$ which fix the image of $P$ in $A$ elementwise. Let also $\Out_P(A)$ denote the quotient of $\Aut_P(A)$ by the subgroup of inner automorphisms induced by conjugation with the elements of $(A^P)^\times$.

Let $[M]\in\calE(B)$. Then by \cite[7.4, 7.6]{Puig1999}, $M$ has a vertex of the form $\Delta(P,\alpha,P)$ for some $\alpha\in\Aut(P,\calF)$ as a $k(G\times G)$-module.  Let an endopermutation $k[\Delta(P,\alpha,P)]$-module $V$ be a source of $M$.  Boltje, Kessar and Linckelmann show in \cite{BKL2020} that the map sending the class $[M]$ to the pair $(V,\alpha)$ induces a group homomorphism $\Phi$ making the diagram
\begin{equation}
\begin{tikzcd}
1\ar[r]&\Out_P(A)\ar[r]\arrow[equal]{d}&\calE(B) \arrow[rightarrow]{r}{\Phi}&D_k(P,\calF)\rtimes\Out(P,\calF)\\
1\ar[r]&\Out_P(A)\ar[r]&\calT(B)\ar[u]\arrow[rightarrow]{r}{\Phi}&\Out(P,\calF)\ar[u]
\end{tikzcd}
\end{equation}
with exact rows commute where upward maps are inclusions.  They also show in \cite{BKL2020} that if the map $\Phi$ maps $\calT(B)$ onto $\Out(P,\calF)$, then $\Phi$ maps $\calE(B)$ to $D_k^t(P,\calF)\rtimes\Out(P,\calF)$, see \cite[Theorem~1.1]{BKL2020}. Hence they raise the question of the surjectivity of the map
\begin{equation}
\Phi:\calT(B)\to \Out(P,\calF)
\end{equation}
and also give examples of blocks for which $\Phi$ is not surjective (see \cite[Section~7]{BKL2020}). Our aim in this paper is to show that this map is split surjective for principal $2$-blocks with dihedral and generalised quaternion defect groups.
\begin{theorem}\label{thm mainthm}
Let $k$ be an algebraically closed field of characteristic $2$. Let $G$ be a finite group with a dihedral or generalised quaternion Sylow $2$-subgroup $P$,  let $B$ be the principal block of $kG$ and let $\calF=\calF_P(G)$. Then for any $\alpha\in\Aut(P,\calF)$, the Scott module $\mathrm{Sc}(G\times G,\Delta(P,\alpha,P))$ induces a splendid Morita auto-equivalence of $B$. In particular, the map
\begin{align*}
\Phi : \pperm(B,B) \rightarrow \mathrm{Out}(P,\calF)
\end{align*}
is split surjective and one has
\begin{align*}
\calT(B)\cong \Out_P(A)\rtimes \Out(P,\calF)\,.
\end{align*}
\end{theorem}

To prove our main theorem, it suffices to consider the groups $G$ listed in  \cite[Theorem~1.1]{KoshitaniLassueur2020A} for dihedral defect group case, and in  \cite[Theorem~1.1]{KoshitaniLassueur2020B} for generalised quaternion defect group case.  Koshitani and Lassueur show in particular that if $G$ is one of these groups and if $B$ is the principal block of $kG$ with a defect group $P$, then the Scott module $\Sc(G\times G,\Delta(P))$ induces a splendid Morita auto-equivalence of $B$. We will follow the blueprint of \cite{KoshitaniLassueur2020A} and \cite{KoshitaniLassueur2020B} to show that actually the Scott module $\Sc(G\times G,\Delta(P,\alpha,P))$, for any $\alpha\in\Aut(P,F)$ induces a splendid Morita auto-equivalence of $B$. The fact that the map $\alpha$ may not be the identity map makes some of the preliminary results and proofs more technical. We will try to write as much as possible to explain these technicalities clearly and as little as possible not to repeat the proofs in \cite{KoshitaniLassueur2020A} and \cite{KoshitaniLassueur2020B}.

\section{Proof of Main Theorem}

In this section we prove Theorem~\ref{thm mainthm}. Let $k$ be an algebraically closed field of characteristic $p>0$.

\begin{lemma}\label{lem fusionlemma}
Let $G$ and $G'$ be two finite groups with a common Sylow $p$-subgroup $P$, and assume that $\calF:=\calF_P(G)=\calF_P(G')$. Let $\alpha\in\Aut(P,\calF)$. Then $\calF_{\Delta(P,\alpha,P)}(G\times G') \cong \calF$. In particular, $\calF_{\Delta(P,\alpha,P)}(G\times G')$ is saturated. 
\end{lemma}
\begin{proof}
One shows that the map
\begin{align*}
\calF&\to \calF_{\Delta(P,\alpha,P)}(G\times G')\\
Q&\mapsto \Delta(\alpha(Q),\alpha,Q)\\
(i_g:Q\to R)&\mapsto \left((i_{g'},i_g): \Delta(\alpha(Q),\alpha,Q)\to \Delta(\alpha(R),\alpha,R)\right)
\end{align*}
where $g'\in G$ with $\alpha\circ i_g\alpha^{-1}=i_{g'}:\alpha(Q)\to\alpha(R)$ is an isomorphism. 
\end{proof}

\begin{lemma}\label{lem p-nilpotent}
Let $G$ and $G'$ be finite $p$-nilpotent groups with a common Sylow $p$-subgroup $P$. Let $\alpha\in\Aut(P)$. Then $\Sc(G\times G', \Delta(P,\alpha,P))$ induces a Morita equivalence between $B_0(kG)$ and $B_0(kG')$.
\end{lemma}
\begin{proof}
Let $M:=\Sc(G\times G', \Delta(P,\alpha,P))$, $B:=B_0(kG)$ and $B':=B_0(kG')$. By definition of Scott modules we have
\begin{align*}
M\mid 1_B\cdot \Ind^{G\times G'}_{\Delta(P,\alpha,P)}k\cdot 1_{B'}=B\otimes_{kG}\Ind^{G\times G'}_{\Delta(P,\alpha,P)}k\otimes_{kG'}B'=:N\,.
\end{align*}
One has
\begin{align*}
N\otimes_{B'}N^o&=\left(B\otimes_{kG}\Ind^{G\times G'}_{\Delta(P,\alpha,P)}k\otimes_{kG'}B'\right)\otimes_{B'}\left(B'\otimes_{kG'}\Ind^{G'\times G}_{\Delta(P,\alpha^{-1},P)}k\otimes_{kG}B\right)\\&
\cong B\otimes_{kG}\Ind^{G\times G'}_{\Delta(P,\alpha,P)}k\otimes_{kG'}B'\otimes_{kG'}\Ind^{G'\times G}_{\Delta(P,\alpha^{-1},P)}k\otimes_{kG}B\\&
\cong B\otimes_{kG}\Ind^{G\times G'}_{\Delta(P,\alpha,P)}k\otimes_{kG'}kP\otimes_{kG'}\Ind^{G'\times G}_{\Delta(P,\alpha^{-1},P)}k\otimes_{kG}B\\&
\cong B\otimes_{kG}\Ind^{G\times G}_{\Delta(P)}k\otimes_{kG}B\\&
\cong B\otimes_{kP} B\\&
\cong B
\end{align*}
as $(kG,kG)$-bimodules. Similar calculations show that $N^o\otimes_B N\cong B'$. This shows that the bimodule $N$ induces a splendid Morita equivalence between $B$ and $B'$.  In particular, $N$ is indecomposable and hence $M=N$. The result follows. 
\end{proof}

The following is a slight generalization of \cite[Lemma~3.4]{KoshitaniLassueur2020A}.
\begin{lemma}\label{lem stableMoritaimplications}
Let $G$ and $G'$ be finite groups with a common Sylow $p$-subgroup $P$ such that $\calF=\calF_P(G)=\calF_P(G')$.  Let $\alpha\in\Aut(P,\calF)$ and set $M :=\Sc(G\times G',\Delta(P,\alpha,P))$, $B := B_0(kG)$ and $B' := B_0(kG')$. If $M$ induces a stable equivalence of Morita type between $B$ and $B'$, then the following holds:

\smallskip
{\rm (a)} $k_G\otimes_B M =k_{G'}$

\smallskip
{\rm (b)} If $U$ is an indecomposable $p$-permutation $kG$-module with vertex $1\neq Q\le P$, then $U\otimes_B M$ has, up to isomorphism, a unique indecomposable direct summand $V$ with vertex $Q$, and again $V$ is a $p$-permutation module.

\smallskip
{\rm (c)} For any $Q\le P$, $\Sc(G,Q)\otimes_B M=\Sc(G',Q)\oplus \mathrm{proj}$

\smallskip
{\rm (d)} For any $Q\le P$, $\Omega_Q(k_G)\otimes_B M=\Omega_Q(k_{G'})\oplus \mathrm{proj}$
\end{lemma}
\begin{proof}
First of all, similar to \cite[Lemma~3.3]{KoshitaniLassueur2020A}, one shows that for an indecomposable $\Delta(P,\alpha,P)$-projective $p$-permutation $k(G\times G')$-module $M$ and a subgroup $Q$ of $P$ the following conditions are equivalent:

\smallskip
{\rm (i)} $\Sc(G',Q) \mid k_G\otimes_{kG} M$.

\smallskip
{\rm (ii)} $\Sc(G\times G',\Delta(\alpha(Q),\alpha,Q))\mid M$.

Indeed, one shows that 
\begin{align*}
k\otimes_{kG}\Ind_{\Delta(\alpha(Q),\alpha,Q)}^{G\times G'}k\cong \Ind_{1\times G}^{1\times G}k\otimes_{kG}\Ind_{\Delta(\alpha(Q),\alpha,Q)}^{G\times G'}k\cong \Ind_Q^{G'}k
\end{align*}
and the rest of the proof is similar to the proof of \cite[Lemma~3.2]{KoshitaniLassueur2020A}. Now part {\rm (a)} follows from the equivalence of {\rm (i)} and {\rm (ii)} applied to the case $Q=P$ and \cite[Theorem~2.1]{KoshitaniLassueur2020A}.  Proofs of parts {\rm (b)-(d)} are similar to the proofs of parts {\rm (b)-(d)} in \cite[Lemma~3.4]{KoshitaniLassueur2020A}.
\end{proof}

In what follows, we assume that $p=2$.

\begin{proposition}\label{prop stableMorita}
Let $G$ be a finite group with a Sylow $2$-subgroup $P$ which is a dihedral group of order at least $8$ and set $\calF=\calF_P(G)$.  Let $\alpha\in\Aut(P,\calF)$. Then the Scott module $\Sc(G\times G,\Delta(P,\alpha,P))$ induces a stable Morita auto-equivalence of $B_0(kG)$.
\end{proposition}
\begin{proof}
Set $M:=\Sc(G\times G,\Delta(P,\alpha,P))$ and $B=B_0(kG)$.  Further, for each subgroup $Q\le P$, set $B_Q:=B_0(kC_G(Q))$. Similar to \cite[Lemma~4.1]{KoshitaniLassueur2020A} one first shows that the following are equivalent: 

\smallskip
{\rm (i)} $M$ induces a stable Morita auto-equivalence of $B$.

\smallskip
{\rm (ii)} For every cyclic subgroup $Q\le P$ of order $p$, the bimodule $M(\Delta(\alpha(Q),\alpha,Q))$ induces a Morita equivalence between $B_{\alpha(Q)}$ and $B_Q$.

Now, since we can assume that $G$ is one of the groups listed in \cite[Theorem~1.1]{KoshitaniLassueur2020A}, it follows that $P$ has either one or two $G$-conjugacy classes of involutions. It suffices to show that for an involution $t\in P$, the bimodule $M\left(\langle \alpha(t),t\rangle\right)$ induces a Morita equivalence between $B_{\langle\alpha(t)\rangle}$ and $B_{\langle t\rangle}$.    

First let $z$ be an involution in $Z(P)$. Then $P$ is a Sylow $2$-subgroup of $C_G(z)$ and $C_G(\alpha(z))$ which are both $2$-nilpotent. So by Lemma~\ref{lem p-nilpotent}, the bimodule $$M_z:=\Sc\left(C_G(\alpha(z))\times C_{G}(z),\Delta(P,\alpha,P)\right)$$ induces a Morita equivalence between $B_{\langle\alpha(z)\rangle}$ and $B_{\langle z\rangle}$. By adopting the proof of \cite[Lemma~3.2]{KoshitaniLassueur2020A}, and noting that the fusion system $\calF_{\Delta(P,\alpha,P)}(G\times G)$ is saturated by Lemma~\ref{lem fusionlemma}, one shows that
\begin{align*}
M_z | M(\langle \alpha(z),z\rangle)\,.
\end{align*}
Note that the module $M$ is Brauer indecomposable. Indeed, every subgroup of $\Delta(P,\alpha,P)$ is of the form $\Delta(\alpha(Q),\alpha,Q)$ where $Q\le P$, and one has $C_{G\times G}(\Delta(\alpha(Q),\alpha,Q))=C_G(\alpha(Q))\times C_G(Q)$. As in the proof of \cite[Corollary~4.4]{KoshitaniLassueur2020A}, it follows that $C_{G\times G}(\Delta(\alpha(Q),\alpha,Q)$ is $2$-nilpotent if $Q\neq 1$, and hence $M$ is Brauer indecomposable by \cite[Theorem~1.3]{KoshitaniLassueur2020A}. This shows that
\begin{align*}
M_z = M(\langle \alpha(z),z\rangle)\,.
\end{align*}

{\em Case 1:} Assume that all involutions on $P$ are $G$-conjugate and let $t\in P$ be an involution. Then $t$ is conjugate to an involution $z\in Z(P)$, i.e., $t=z^g$ for some $g\in G$. Since $\alpha(t)$ is again an involution, we have $\alpha(t)=\alpha(z)^{g'}$ for some $g'\in G$, as well. So,
\begin{align*}
M\left(\langle \alpha(t),t\rangle\right)&=M\left(\langle \alpha(z),z\rangle^{(g',g)}\right)=M_z^{(g',g)}=\Sc(C_G(\alpha(z)^{g'})\times C_G(z^g),\Delta(P,\alpha,P)^{(g',g)})\\&
=\Sc(C_G(\alpha(t))\times C_G(t),\Delta(P^{g'},i_{g'^{-1}}\alpha i_g,P^g))\,.
\end{align*}
Now $P^{g'}$ and $P^g$ are Sylow $2$-subgroups of $C_G(\alpha(t))$ and $C_G(t)$, respectively. So by Lemma~\ref{lem p-nilpotent}, it follows that
\begin{align*}
M\left(\langle \alpha(t),t\rangle\right)=\Sc(C_G(\alpha(t))\times C_G(t),\Delta(P^{g'},i_{g'^{-1}}\alpha i_g,P^g))
\end{align*}
induces a Morita equivalence between $B_{\langle\alpha(t)\rangle}$ and $B_{\langle t\rangle}$.

{\em Case 2:} Assume now that $P$ has exactly two $G$-conjugacy classes of involutions. Let $t$ and $z$ be two non-conjugate involutions with $z\in Z(P)$. By the proof of \cite[Lemma~4.5]{KoshitaniLassueur2020A}, the group $\langle t\rangle$ is a fully $\calF_P(G)$-normalized subgroup of $P$. Hence, by Lemma~\ref{lem fusionlemma}, the group $\Delta(\langle \alpha(t)\rangle,\alpha,\langle t\rangle)$ is a fully $\calF_{\Delta(P,\alpha,P)}(G\times G)$-normalized subgroup of $\Delta(P,\alpha,P)$. Since also $M$ is Brauer indecomposable, by \cite[Theorem~1.3]{IshiokaKunugi2017} it follows that
\begin{align*}
M\left(\Delta(\langle \alpha(t)\rangle,\alpha,\langle t\rangle)\right)\cong \Sc\left(N_{G\times G}(\Delta(\langle \alpha(t)\rangle,\alpha,\langle t\rangle), N_{\Delta(P,\alpha,P}(\Delta(\langle \alpha(t)\rangle,\alpha,\langle t\rangle)))\right)\,.
\end{align*}
Since $|\Delta(\langle \alpha(t)\rangle,\alpha,\langle t\rangle)|=2$, one has
\begin{align*}
N_{G\times G}(\Delta(\langle \alpha(t)\rangle,\alpha,\langle t\rangle)=C_G(\alpha(t))\times C_G(t)
\end{align*}
and 
\begin{align*}
N_{\Delta(P,\alpha,P)}(\Delta(\langle \alpha(t)\rangle,\alpha,\langle t\rangle)=\Delta(\alpha(C_P(t)),\alpha,C_P(t)))\,.
\end{align*}
Therefore, one has
\begin{align*}
M\left(\Delta(\langle \alpha(t)\rangle,\alpha,\langle t\rangle)\right)\cong \Sc\left(C_G(\alpha(t))\times C_G(t),\Delta(\alpha(C_P(t)),\alpha,C_P(t)))\right)\,.
\end{align*}
Now, the proof in this case is similar to the proof of Case 2 in \cite[Proposition~4.6]{KoshitaniLassueur2020A}, with \cite[Lemma~4.5]{KoshitaniLassueur2020A} is replaced by the isomorphism above.
\end{proof}

{\em Proof of Theorem~\ref{thm mainthm}:} First assume that $P=D_{2^n}$ is a dihedral group. We can assume that $G$ is one of the groups listed in \cite[Theorem~1.1]{KoshitaniLassueur2020A}. By Proposition~\ref{prop stableMorita}, the bimodule $\Sc(G\times G,\Delta(P,\alpha,P))$ induces a stable Morita auto-equivalence of $B$. To show that this is indeed a Morita equivalence, one follows the steps in \cite[Section~5]{KoshitaniLassueur2020A} with \cite[Lemma~3.4]{KoshitaniLassueur2020A} is replaced by Lemma~\ref{lem stableMoritaimplications}. 

Now assume that $P=Q_{2^n}$ is a generalised quaternion group. We can assume that the group $G$ is one of the groups listed in \cite[Theorem~1.1(a)]{KoshitaniLassueur2020B}. Hence $Z:=Z(G)=Z(P)$ is a group of order $2$. Write $M:=\mathrm{Sc}(G\times G,\Delta(P,\alpha,P))$, $\overline{G}:=G/Z$, $\overline{P}:=P/Z$ and $\overline{M}:=k\overline{G}\otimes_{kG}M\otimes_{kG} k\overline{G}$. Note that the group $\overline{P}=D_{2^{n-1}}$ is dihedral and $\overline{G}$ is one of the groups listed in \cite[Theorem~1.1]{KoshitaniLassueur2020A}.  

Since $Z$ is the center of $P$, the map $\alpha$ induces a map $\overline{\alpha}:\overline{P}\to \overline{P}$. Moreover, again $Z$ is central in $G$, one has $\overline{\alpha}\in\Aut(\overline{P},\overline{\calF})$ where $\overline{\calF}=\calF_{\overline{P}}(\overline{G})$. Now we claim that $\overline{M}=\mathrm{Sc}(\overline{G}\times \overline{G}, \Delta(\overline{P},\overline{\alpha},\overline{P}))$ which is essentially \cite[Lemma~3.1]{KoshitaniLassueur2020B}. We include a short proof for the convenience of reader.

Since the group $\Delta(Z,\alpha,Z)$ is central in $G\times G$, it acts trivially on $\Ind_{\Delta(P,\alpha,P)}^{G\times G}k$ and hence on $M$.  One has
\begin{align*}
\overline{M}\cong M\otimes_{k(G\times H} k(\overline{G}\times \overline{H})& \mid \Ind_{\Delta(P,\alpha,P)}^{G\times G}k \otimes_{k(G\times H} k(\overline{G}\times \overline{H})\\&
\cong k\otimes_{\Delta(P,\alpha,P)} k(\overline{G}\times \overline{H})\\&
\cong k\otimes_{\Delta(\overline{P},\overline{\alpha},\overline{P})} k(\overline{G}\times \overline{H})=\Ind_{\Delta(\overline{P},\overline{\alpha},\overline{P})}^{\overline{G}\times \overline{H}} k
\end{align*}
as right $k(\overline{G}\times \overline{H})$-modules. This proves our claim. 

Now note that 
\begin{align*}
\Delta(P,\alpha,P)\cap ({1}\times G)=\Delta(P,\alpha,P)\cap (G\times {1})={1}
\end{align*}
and 
\begin{align*}
Z\times Z\le (Z\times {1})\Delta(P,\alpha,P)=({1}\times Z)\Delta(P,\alpha,P)
\end{align*}
since $(z,z')=(z\alpha(z')^{-1},1)(\alpha(z'),z')$, $(z,z')=(1,z'(\alpha^{-1}(z))^{-1})(z,\alpha^{-1}(z))$ and $(1,z)(\alpha(u),u)=(\alpha(z^{-1}),1)(\alpha(zu),zu)$.  Hence \cite[Lemma~10.2.11]{Rouquier98} applied replacing splendid complex with bimodule inducing splendid Morita equivalence, as in the proof of \cite[Lemma~3.2]{KoshitaniLassueur2020B}, implies that $M$ induces a splendid Morita auto-equivalence of $B$ if and only if $\overline{M}$ induces a splendid Morita auto-equivalence of $\overline{B}$. The result follows from the first part.
\qed

\centerline{\rule{5ex}{.1ex}}
\begin{flushleft}
\c{C}isil Karag{\"u}zel, Department of Mathematics, Bilkent University, 06800 Ankara, Turkey.\\
{\tt cisil@bilkent.edu.tr}\vspace{1ex}\\
Deniz Y\i lmaz, Department of Mathematics, Bilkent University, 06800 Ankara, Turkey.\\
{\tt d.yilmaz@bilkent.edu.tr}
\end{flushleft}
\end{document}